\documentclass[12pt,a4paper]{article}
\usepackage{fullpage}
\usepackage{authblk}

\usepackage{mathrsfs}
\usepackage{amsfonts}
\usepackage{amssymb}
\usepackage{bm}
\usepackage{enumerate}
\usepackage{tikz}
\usepackage{graphicx}
\usepackage{stfloats}
\usepackage{amsmath}

\usepackage{mathrsfs}
\usepackage{amsfonts}
\usepackage{amssymb}
\usepackage{bm}
\usepackage{enumerate}
\usepackage{tikz}
\usepackage{graphicx}
\usepackage{stfloats}
\usepackage{amsmath}
\usepackage[colorlinks,
            linkcolor=red,
            anchorcolor=blue,
            citecolor=green
            ]{hyperref}

\newtheorem{theorem}{Theorem}[section]
\newtheorem{definition}[theorem]{Definition}
\newtheorem{remark}[theorem]{Remark}

\newtheorem{lemma}[theorem]{Lemma}
\newtheorem{example}[theorem]{Example}

\numberwithin{equation}{section}

\def \H {\mathbb{H}}
\def \X {\mathbf{X}}
\def\i{{\bf i}}
\def\j{{\bf j}}
\def\k{{\bf k}}
\def\om{\omega}

\def\e{{\bf e}}
\def\u{{\bm \mu}}

\def\erf{{\mathrm{erf}}}
\def \La {L^1(\mathbb{R}^2,\H)}
\def \Lb {L^2(\mathbb{R}^2,\H)}
\def\Sc{{\mathrm{Sc}}}
\def\Vec{{\mathrm{Vec}}}

\newenvironment{proof}{\noindent{\em \textbf{Proof.}}}{\quad \hfill$\Box$\vspace{2ex}}

\title{Generalized Sampling Expansions Associated with Quaternion Fourier Transform}

\author{Dong Cheng\thanks{chengdong720@163.com}  }

\author{Kit Ian Kou\thanks{Corresponding author: kikou@umac.mo}}

\affil{\normalsize{Department of Mathematics, Faculty of Science and Technology, University of Macau, Macao, China}}
\date{}

\begin{document}

   \maketitle

\begin{abstract}
\normalsize

Quaternion-valued signals along with quaternion Fourier transforms (QFT) provide an effective framework for  vector-valued signal and  image  processing. However, the sampling theory of quaternion valued signals has not been well developed.  In this paper, we present the generalized sampling expansions  associated with  QFT by using the generalized translation and convolution. We show that  a  $\sigma$-bandlimited  quaternion valued signal in QFT sense can be reconstructed from the samples of output signals of $M$ linear systems based on QFT. Quaternion linear canonical transform (QLCT) is a generalization of QFT with six parameters. Using the relationship between QFT, we derive the sampling formula for $\sigma$-bandlimited quaternion-valued signal in QLCT sense.  Examples are given to  illustrate our results.
\end{abstract}

\begin{keywords}
Quaternion Fourier transform; quaternion linear canonical transform; sampling expansions; generalized translation; convolution theorem
\end{keywords}

%43A25; 42A38; 92A20; 51A40; 45E10

\section{Introduction}\label{S1}

{Generalized} sampling expansions  (GSE) developed by Papoulis \cite{papoulis1977generalized} indicates that a $\sigma$-bandlimited signal can be reconstructed from the  samples of output signals of $M$ linear   systems. Namely,
\begin{equation*}
f(t)=\sum^\infty_{n=-{\infty}}\sum_{m=1}^M g_m(nT)y_m(t-nT)
\end{equation*}
where $g_m~(1\leq m\leq M)$ are output signals of $M$ linear  systems, $y_m~(1\leq m\leq M)$ are determined by a linear simultaneous equations whose  coefficients are generated by system functions.
Some classical sampling expansions, for instance, Shannon sampling expansions are special cases of  Papoulis' result by choosing specific systems.

Over the years, the GSE has been extended in different ways. Hoskins and Pinto \cite{hoskins1984generalized} extended
the GSE to   bandlimited  distribution functions.  A multidimensional extension of   GSE was   introduced by Cheung \cite{cheung1993multidimensional} for real-valued functions. While, Wei, Ran and Li  \cite{wei2010generalized}  presented the GSE with generalized integral transformation, such as   fractional Fourier transform. In this paper, higher-dimensional extension of GSE to quaternion-valued functions are studied. By powerful modelling of rotation and orientation, quaternion have shown advantages in physical and engineering applications such as computer graphics \cite{shoemake1985animating,hanson1995quaternion} and robotics \cite{yun2006design}. Furthermore, QFT has been regarded as a useful analysis tool in color  image and signal processing   \cite{sangwine1996fourier,bulow1999hypercomplex,bihan2003quaternion,le2004singular,ell2007hypercomplex}  in recently years.  Therefore, it is desirable to define a system based on QFT to analyze quaternion-valued signals. Moreover, it is worthwhile and interesting to
investigate the GSE using the samples of output signals of $M$ linear systems based on QFT.
However, for the non-commutativity of the quaternion multiplication, the desirable shift property of classical Fourier transform is no longer available for QFT. Meanwhile, an crucial tool in signal processing  called     convolution theorem    does not hold for QFT as well. The purpose of this paper is to overcome these problems and investigate the GSE.
In this paper, we  propose a novel translation of quaternion-valued signals and apply it to deduce the  convolution theorem of QFT. More importantly, we present the GSE associated with QFT by proposed translation and convolution.

The rest of the paper is organized as follows. In the next section, we   review    QFT  and some of its properties such as   Plancherel theorem. Section \ref{S3}  proposes a new  translation and its corresponding  convolution theorem. In Section \ref{S4}, we present the generalized sampling expansion of bandlimited quaternion-valued signals in the sense of QFT. In Section \ref{S5}, examples are presented to  illustrate our results. In Section \ref{S6}, we further discuss the sampling formula for $\sigma$-bandlimited  quaternion valued signal in QLCT sense.

\section{Preliminary}\label{S2}
\subsection{Quaternion  algebra}
Let's recall quaternion algebra $ \H :=\{q=q_0+\i q_1+\j q_2+\k q_3:~q_0,q_1,q_2,q_3\in\mathbb{R}\},$ where the imaginary  elements $\i$, $\j$ and $\k$ obey $\i^2=\j^2=\k^2=\i\j\k=-1$. For every quaternion $q=q_0+\underline{q}$, $\underline{q}=\i q_1+\j q_2+\k q_3$, the scalar and vector parts of $q$, are  defined as $\Sc(q)=q_0$ and $\Vec(q)=\underline{q}$, respectively. If $q=\Vec(q)$, then $q$ is called pure imaginary quaternion.
The quaternion conjugate is defined by $\overline{q}=q_0-\underline{q}=q_0-\i q_1-\j q_2-\k q_3$, and the norm $|q|$ of $q$ defined as
$|q|^2={q\overline{q}}={\overline{q}q}=\sum_{m=0}^{m=3}{q_m^2}$.
Then we have
\begin{equation*}
  \overline{\overline{q}}=q,~~~\overline{p+q}=\overline{p}+\overline{q},~~~\overline{pq}=\overline{q}~\overline{p},~~~|pq|=|p||q|,~~~~\forall p,q\in\H.
\end{equation*}
Using the conjugate and norm of $q$, one can define the inverse of $q\in\H\backslash\{0\}$ as $q^{-1}=\overline{q}/|q|^2$.

The quaternion exponential function $\e^{q}$ is defined by means of an infinite series as $\e^{q}:=\sum_{n=0}^\infty \frac{q^n}{n!}.$
Analogous to the complex case one may derive a closed-form representation:
$\e^{q}=\e^{q_0}(\cos|\underline{q}|+\frac{\underline{q}}{|\underline{q}|}\sin|\underline{q}|).$

Let $\X$ be a Lebesgue measurable subset of  $\mathbb{R}^2$, the left   $\H$-module $L^p(\X,\H)(p=1,2)$  consists of all $\H$-valued  functions whose $p$th power is Lebesgue integrable on $\X$. The left quaternionic inner product of $f,g\in L^2(\X,\H)$ is defined by
\begin{equation*}
 \langle f,g \rangle:=\int_{\X}f(x_1,x_2)\overline{g(x_1,x_2)}dx_1dx_2.
\end{equation*}
In  fact, $L^2(\X,\H)$ is a left quaternionic Hilbert space with inner product $\langle \cdot,\cdot\rangle $ (see \cite{brackx1982clifford}). Therefore, if $\{e_n\}$ is  an orthonormal basis of $L^2(\X,\H)$, then
\begin{equation}\label{Plancherel series}
  \langle f,g\rangle=\sum_n\langle f,e_n\rangle\langle e_n,g\rangle
\end{equation}
 holds for all $f,g\in L^2(\X,\H)$. This is a desirable property of  quaternionic Hilbert space \cite{ghiloni2013continuous}.

 \subsection{Quaternion Fourier transform }
The quaternion
Fourier transform (QFT) was first introduced by Ell to  analyze   partial differential equations \cite{ell1993quaternion}. Since then,  QFT were applied to color image processing effectively \cite{sangwine1996fourier,bulow1999hypercomplex,bihan2003quaternion,ell2007hypercomplex}.    There are different types of QFT \cite{ell2013quaternion} due to the non-commutativity of the quaternion multiplication . In  \cite{cheng2016properties},  the authors investigated the properties of distinct types of QFT thoroughly,  especially the following  right-sided QFT.
\begin{definition}[QFT]
For every $f\in\La$, the right-sided  QFT of $f$ is defined by
\begin{equation*}
 (\mathcal{F}f)(\om_1,\om_2) :=\frac{1}{2\pi} \int_{\mathbb{R}^2}f(x_1,x_2)\displaystyle\e^{-\i\om_1x_1}\e^{-\j\om_2x_2}dx_1dx_2.
\end{equation*}
\end{definition}

If $\mathcal{F}f$ is also in $\La$,    the inversion QFT formula (see \cite{hitzer2007quaternion,cheng2016properties}) holds,  that is
\begin{equation*}
f(x_1,x_2)=(\mathcal{F}^{-1}\mathcal{F}f)(x_1,x_2):= \frac{1}{2\pi} \int_{\mathbb{R}^2}(\mathcal{F}f)(\om_1,\om_2)\e^{\j \omega_2x_2}\e^{\i \omega_1x_1}d\om_1\om_2,
\end{equation*}
for almost every $(x_1,x_2)\in \mathbb{R}^2$.
 By Plancherel theorem (see \cite{hitzer2007quaternion,cheng2016properties}), the QFT can be extended to $\Lb$. As an operator on $\Lb$, the QFT $\Psi$ is a
 bijection and  the  Parseval's identity $\|\Psi f\|_2=\|f\|_2$ holds.
 \begin{remark}
Since $\Psi(\Psi^{-1})$ coincides with  $ \mathcal{F}(\mathcal{F}^{-1})$ in $\La\cap \Lb$. For simplicity of notations, in the following, by capital letter $F$, we mean the QFT of $f\in \La \cup \Lb$ if no otherwise specified.
\end{remark}

\section{Generalized translation and  convolution}\label{S3}
A generalized translation  related to  the general integral transform with kernel $K(\om,t)$   was introduced in \cite{marks2012advanced}.  Motivated by this study, we define the generalized translation to the quaternion-valued signals.

\begin{definition}\label{def generalized translation}
Let $f\in\La \cup\Lb$ and $F\in\La$. The generalized translation related to QFT is defined by
\begin{equation}\label{generalized translation}
 f( x_1\ominus y_1, x_2\ominus  y_2):= \frac{1}{2\pi}\int_{\mathbb{R}^2}\e^{-\i\om_1y_1}\e^{-\j\om_2y_2}F(\om_1,\om_2)
  \e^{\j\om_2x_2}\e^{\i\om_1x_1}d\om_1d\om_2.
\end{equation}
\end{definition}

\begin{remark}
In  the complex case, the 2D Fourier transform of  $f(x_1-y_1, x_2- y_2)$ with respect to $ (x_1,x_2)$ is $ \e^{-\i\om_1y_1}\e^{-\i\om_2y_2}\widehat{f}(\om_1,\om_2)$, where $\widehat{f}(\om_1,\om_2)$ is the  2D  Fourier transform of complex-valued function $f$. Therefore, the generalized translation $f( x_1\ominus y_1, x_2\ominus  y_2)$, in some sense, is an analogue of $ f(x_1-y_1, x_2- y_2)$. Moreover, $f( x_1\ominus y_1, x_2\ominus  y_2)$  coincides with $f(x_1-y_1, x_2- y_2)$ for some special $f(x_1,x_2)$ (see Example \ref{converges fast}).
\end{remark}

Suppose that $h\in\La \cup\Lb$, $H\in\La$ and $f\in \La$, then $$\int_{\mathbb{R}^4}|f(y_1,y_2)H(\om_1,\om_2)|dy_1dy_2d\om_1d\om_2<\infty.$$
Therefore, by Fubini's Theorem,  we have
\begin{align}
% \nonumber to remove numbering (before each equation)
  &~~~~\frac{1}{2\pi}\int_{\mathbb{R}^2}f(y_1,y_2) ( x_1\ominus y_1, x_2\ominus  y_2)dy_1dy_2 \nonumber \\
   &=   \frac{1}{4\pi^2}\int_{\mathbb{R}^2}dy_1dy_2f(y_1,y_2)  \int_{\mathbb{R}^2} \e^{-\i\om_1y_1}\e^{-\j\om_2y_2}H(\om_1,\om_2)
   \e^{\j\om_2x_2}\e^{\i\om_1x_1}d\om_1d\om_2 \nonumber \\
     &=  \frac{1}{2\pi} \int_{\mathbb{R}^2} F(\om_1,\om_2)H(\om_1,\om_2)
   \e^{\j\om_2x_2}\e^{\i\om_1x_1}d\om_1d\om_2 \nonumber\\
   &=   (\mathcal{F}^{-1}G)(x_1,x_2),
\end{align}
where $G=FH$. Thus it is reasonable to define the following generalized convolution.
\begin{definition}
Let $f,h, G=FH \in\La \cup\Lb$. The convolution of $f$ and $h$ is defined by
\begin{equation*}
  (f\star h)(x_1,x_2):= (\mathcal{F}^{-1}G)(x_1,x_2) .
\end{equation*}
\end{definition}

Obviously, The QFT of $f\star h$ is $FH$.

\begin{theorem}\label{CONvo}
If any of the following conditions is satisfied.
\begin{enumerate}
  \item If $h\in\La \cup\Lb$, $H\in\La$ and $f\in \La$.
  \item If $h\in\La \cup\Lb$, $H\in\La\cap \Lb$ and $f\in \Lb$.
\end{enumerate}
 Then
\begin{equation*}
    (f\star h)(x_1, x_2)=\frac{1}{2\pi}\int_{\mathbb{R}^2}f(y_1,y_2) ( x_1\ominus y_1, x_2\ominus  y_2)dy_1dy_2.
  \end{equation*}
\end{theorem}

\begin{proof}
The first case is obviously true. We now prove the second case.
Since  $f\in \Lb$, then $F\in \Lb$ by Plancherel theorem. Moreover,  there is a sequence $\{f_n\}$ in $\La \cap\Lb$ converging to $f$ in $L^2$ norm such
that $F=\mathop{\mathrm{l.i.m.}}\limits_{n\rightarrow\infty}F_n= \mathop{\mathrm{l.i.m.}}\limits_{n\rightarrow\infty}\mathcal{F}f_n$. Note that $H\in  \Lb$. Therefore, by H$\mathrm{\ddot{o}}$lder inequality, $FH\in \La$ and
\begin{align}
  &~~~~ (f\star h)(x_1, x_2) \nonumber \\
   &=  \frac{1}{2\pi}\int_{\mathbb{R}^2}F(\om_1,\om_2)H(\om_1,\om_2)
   \e^{\j\om_2x_2}\e^{\i\om_1x_1}d\om_1d\om_2  \nonumber\\
   &= \lim_{n\rightarrow \infty}\frac{1}{2\pi}\int_{\mathbb{R}^2} F_n(\om_1,\om_2)H(\om_1,\om_2)\e^{\j\om_2x_2}\e^{\i\om_1x_1}d\om_1d\om_2 \nonumber \\
   &=   \lim_{n\rightarrow \infty}\frac{1}{4\pi^2}\int_{\mathbb{R}^4} dy_1dy_2f_n(y_1,y_2)\e^{-\i\om_1y_1}\e^{-\j\om_2y_2} H(\om_1,\om_2)\e^{\j\om_2x_2}\e^{\i\om_1x_1}d\om_1d\om_2  \nonumber\\
  &=   \lim_{n\rightarrow \infty}\frac{1}{2\pi}\int_{\mathbb{R}^2}f_n(y_1,y_2) ( x_1\ominus y_1, x_2\ominus  y_2)dy_1dy_2\nonumber\\
   &=  \frac{1}{2\pi}\int_{\mathbb{R}^2}f(y_1,y_2) ( x_1\ominus y_1, x_2\ominus  y_2)dy_1dy_2.
\end{align}
The interchange of integral and limit  is permissible for the continuity of   inner product.
\end{proof}

\section{The GSE associated with QFT}\label{S4}
We  give a definition of bandlimited signals in QFT sense.
\begin{definition}[bandlimited]
A signal $f(x_1,x_2)$ is $\sigma$-bandlimited in QFT sense if it can be expressed as
\begin{equation*}
  f(x_1,x_2)=\frac{1}{2\pi}\int_{I}F(\om_1,\om_2)\e^{\j\om_2x_2}\e^{\i\om_1x_1}d\om_1d\om_2
\end{equation*}
where $F\in L^2(I,\H)$ and $I=[-\sigma,\sigma]^2$. For any $\sigma >0$, denote by $\mathbf{B}_{\sigma}^{q}$ the totality of the $\sigma$-bandlimited signals  in QFT sense.
\end{definition}

If $f \in \mathbf{B}_{\sigma}^{q}$, by Plancherel theorem, we have $f\in\Lb$ and the QFT of $f$ is $F$.  In this part, we show that $f$ can be reconstructed from the samples of the inverse QFT of  $M:=m^2$ functions $G_k=F H_k, ~(k=1,2,...,M)$  if $H_k $ satisfy suitable conditions.

Let $T:=\frac{m\pi}{\sigma}$, $c:=\frac{2\sigma}{m}=\frac{2\pi}{T}$ and
\begin{equation}\label{interval}
  I_{n_1n_2}:=[-\sigma+(n_1-1)c ,-\sigma +n_1 c]\times[-\sigma+(n_2-1)c ,-\sigma +n_2 c].
\end{equation}
 Then
\begin{equation}\label{bandlimited original func}
% \nonumber to remove numbering (before each equation)
   f(x_1,x_2) %&=&\sum_{n_1=1}^m\sum_{n_2=1}^m\frac{1}{2\pi} \int_{I_{n_1n_2}}F(\om_1,\om_2)\e^{\j\om_2x_2}\e^{\i\om_1x_1}d\om_1d\om_2  \\
      =  \sum_{n_1=0}^{m-1}\sum_{n_2=0}^{m-1} \frac{1}{2\pi} \int_{I_{11}}F(\om_1+n_1 c,\om_2+n_2 c) \e^{\j(\om_2+n_2 c)x_2}\e^{\i(\om_1+n_1 c)x_1}d\om_1d\om_2.
\end{equation}

To state our results, we need some further notations. Let
   \begin{equation*}
  \left\{
   \begin{array}{l}
 a_{n_1n_2}(\om_1,\om_2) := F(\om_1+(n_1-1)c ,\om_2 +(n_2-1) c),      \\
   b_{n_1n_2}(\om_1,\om_2,x_1,x_2):=\e^{\j(\om_2+(n_2-1) c)x_2}\e^{\i(\om_1+(n_1-1)c)x_1},\\
    r_{n_1n_2}^k(\om_1,\om_2):=H_k(\om_1+(n_1-1)c ,\om_2 +(n_2-1) c).
   \end{array}
  \right.
  \end{equation*}
They form the following vectors or matrices:
   \begin{equation*}
  \left\{
   \begin{array}{l}
 \overrightarrow{F}(\om_1,\om_2) :=(A(1,:),A(2,:),...,A(m,:))^T,    \\
   \overrightarrow{E}(\om_1,\om_2,x_1,x_2):=(B(1,:),B(2,:),...,B(m,:))^T,\\
    \overrightarrow{\underline{H}_k}(\om_1,\om_2):=(R_k(1,:),R_k(2,:),...,R_k(m,:))^T,\\
    \underline{H}(\om_1,\om_2):=( \overrightarrow{\underline{H}_1}(\om_1,\om_2), \overrightarrow{\underline{H}_2}(\om_1,\om_2),..., \overrightarrow{\underline{H}_M}(\om_1,\om_2)),\\
    \overrightarrow{G}(\om_1,\om_2)  := (\widetilde{G}_1(\om_1,\om_2),\widetilde{G}_2(\om_1,\om_2),...,\widetilde{G}_M(\om_1,\om_2))
   = {\overrightarrow{F}(\om_1,\om_2)}^T \underline{H}(\om_1,\om_2),
   \end{array}
  \right.
  \end{equation*}
  where   $A$, $B$, $R_k$ are $m\times m$ matrices with entries $a_{n_1n_2}(\om_1,\om_2)$, $b_{n_1n_2}(\om_1,\om_2,x_1,x_2)$, $r_{n_1n_2}^k(\om_1,\om_2)$ respectively.
 Assume that $\underline{H}$ is invertible for every $(\om_1,\om_2)\in I_{11}$. Denote the inverse of $\underline{H}$ by
 \begin{equation*}
   {\underline{H}^{-1}(\om_1,\om_2)}:=( \overrightarrow{\underline{Q}_1}(\om_1,\om_2); \overrightarrow{\underline{Q}_2}(\om_1,\om_2);...; \overrightarrow{\underline{Q}_M}(\om_1,\om_2))
 \end{equation*}
 where $\overrightarrow{\underline{Q}_k}(\om_1,\om_2)=(Q_k(1,:),Q_k(2,:),...,Q_k(m,:))$ and $Q_k=(q^k_{n_1n_2}(\om_1,\om_2))_{m \times m}$.
Then (\ref{bandlimited original func}) becomes
\begin{equation*}
  f(x_1,x_2)=\frac{1}{2\pi}\int_{I_{11}}\overrightarrow{F}(\om_1,\om_2)\overrightarrow{E}(\om_1,\om_2,x_1,x_2)d\om_1,\om_2.
\end{equation*}
In the complex case, $ \overrightarrow{E}$     only depends on $(x_1,x_2)$ (see \cite{cheung1993multidimensional}).   However, due to the non-commutativity of the quaternion algebra,
$ \overrightarrow{E}$ depends on both $(x_1,x_2)$ and $(\om_1,\om_2)$  in quaternionic case.
\begin{lemma}\label{samples-g}
If $H_k\in L^2(I,\H)$,  then samples of the inverse QFT of  $G_k (k=1,2,...,M)$ can be expressed as:
\begin{equation*}
g_k(n_1T,n_2 T)=\frac{1}{2\pi} \int_{I_{11}}\widetilde{G}_k(\om_1,\om_2)\e^{\j\om_2n_2T}\e^{\i\om_1n_1T}d\om_1d\om_2,
\end{equation*}
where $I_{11}$ is given in (\ref{interval}).
\end{lemma}

\begin{proof}
Since $H_k\in L^2(I,\H)$, then $G_k\in L^2(I,\H)$,
\begin{align*}
% \nonumber to remove numbering (before each equation)
  g_k(x_1,x_2) &=\frac{1}{2\pi} \int_{I}G_k(\om_1,\om_2)\e^{\j\om_2x_2}\e^{\i\om_1x_1}d\om_1d\om_2 \\
  &=\frac{1}{2\pi} \int_{I}F(\om_1,\om_2)H_k(\om_1,\om_2)\e^{\j\om_2x_2}\e^{\i\om_1x_1}d\om_1d\om_2 \\
  &=\sum_{l_1=1}^{m}\sum_{l_2=1}^{m} \frac{1}{2\pi} \int_{I_{11}} a_{l_1l_2}(\om_1,\om_2) r_{l_1l_2}^k(\om_1,\om_2)    b_{l_1l_2}(\om_1,\om_2,x_1,x_2)d\om_1d\om_2\\
    &=  \frac{1}{2\pi} \int_{I_{11}}\sum_{l_1=1}^{m}\sum_{l_2=1}^{m}  a_{l_1l_2}(\om_1,\om_2) r_{l_1l_2}^k(\om_1,\om_2)    b_{l_1l_2}(\om_1,\om_2,x_1,x_2)d\om_1d\om_2
\end{align*}
 and $b_{l_1l_2}(\om_1,\om_2,n_1T,n_2T)=\e^{\j\om_2n_2T}\e^{\i\om_1n_1T}$  which is independent of $l_1l_2$. Therefore
 \begin{align*}
&~~~~ \sum_{l_1=1}^{m}\sum_{l_2=1}^{m}  a_{l_1l_2}(\om_1,\om_2) r_{l_1l_2}^k(\om_1,\om_2)    b_{l_1l_2}(\om_1,\om_2,x_1,x_2)   \\
&={\overrightarrow{F}(\om_1,\om_2)}^T\overrightarrow{\underline{H}_k}(\om_1,\om_2)\e^{\j\om_2n_2T}\e^{\i\om_1n_1T}.
 \end{align*}
It follows that
 \begin{align*}
 % \nonumber to remove numbering (before each equation)
    g_k(n_1T,n_2T)&= \frac{1}{2\pi} \int_{I_{11}}{\overrightarrow{F}(\om_1,\om_2)}^T\overrightarrow{\underline{H}_k}(\om_1,\om_2)\e^{\j\om_2n_2T}\e^{\i\om_1n_1T}d\om_1, \om_2 \\
  &=\frac{1}{2\pi} \int_{I_{11}}\widetilde{G}_k(\om_1,\om_2)\e^{\j\om_2n_2T}\e^{\i\om_1n_1T}d\om_1\om_2.
 \end{align*}
\end{proof}

\begin{lemma}\label{interpolation func}
Suppose that $q^k_{l_1l_2} \in L^2(I_{11},\H)$ and let
\begin{equation*}
 \widetilde{q}^k_{n_1n_2}(\om_1,\om_2) =q^k_{n_1n_2}(\om_1-(n_1-1)c, \om_2-( n_2-1)c)\chi_{I_{n_1n_2}}(\om_1,\om_2)
\end{equation*}
 and
 \begin{equation}\label{QFT-interpolation func}
  Y_k(\om_1,\om_2)=\frac{{T^2}}{2\pi} \sum_{n_1=1}^{m }\sum_{n_2=1}^{m }  \widetilde{q}^k_{n_1n_2}(\om_1,\om_2).
 \end{equation}
 Then for every $ (n_1,n_2)\in \mathbb{Z}^2$, $ \frac{4\pi^2}{T^2} y_k(x_1\ominus n_1T,x_2\ominus n_2 T)$ equals to
% \begin{eqnarray}
%     &~& y_k(x_1\ominus n_1T,x_2\ominus n_2T) \\
%     &=& \frac{1}{2\pi} \int_{I_{11}} \e^{-\i\om_1n_1T} \e^{-\j\om_2n_2T} \sum_{l_1=1}^{m}\sum_{l_2=1}^{m}q^k_{l_1l_2}(\om_1,\om_2)b_{l_1l_2}(\om_1,\om_2,x_1,x_2)d\om_1d\om_2
% \end{eqnarray}
\begin{equation*}\label{lemma2}
   \int_{I_{11}} \e^{-\i\om_1n_1T} \e^{-\j\om_2n_2T}   \sum_{l_1=1}^{m}\sum_{l_2=1}^{m}q^k_{l_1l_2}(\om_1,\om_2)b_{l_1l_2}(\om_1,\om_2,x_1,x_2)d\om_1\om_2
\end{equation*}
where $y_k$ is the inverse QFT of $Y_k$.
\end{lemma}
\begin{proof}
By rewriting  $y_k(x_1\ominus n_1T,x_2\ominus n_2 T)$ in the form of  (\ref{generalized translation}) and  substituting $\e^{\j\om_2x_2}\e^{\i\om_1x_1}$ and $\e^{-\i\om_1n_1T} \e^{-\j\om_2n_2T}$  with  $b_{l_1l_2}(\om_1-(l_1-1)c,\om_2-(l_2-1)c,x_1,x_2)$ and $$\overline{b_{n_1n_2}(\om_1-(l_1-1)c,\om_2-(l_2-1)c,n_1T,n_2T)}$$ respectively,   we obtain \begin{align}\label{lemma2eq}
    &~~~~   \frac{2\pi}{T^2}\int_{\mathbb{R}^2}\e^{-\i\om_1n_1T} \e^{-\j\om_2n_2T} Y_k(\om_1,\om_2)\e^{\j\om_2x_2}\e^{\i\om_1x_1}d\om_1d\om_2\nonumber\\
     &= \sum_{l_1=1}^{m}\sum_{l_2=1}^{m}\int_{\mathbb{R}^2}\e^{-\i\om_1n_1T} \e^{-\j\om_2n_2T}\widetilde{q}^k_{l_1l_2}(\om_1,\om_2)\e^{\j\om_2x_2}\e^{\i\om_1x_1}d\om_1d\om_2\nonumber\\
   &= \sum_{l_1=1}^{m}\sum_{l_2=1}^{m}\int_{I_{l_1l_2}}\overline{b_{n_1n_2}(\om_1-(l_1-1)c,
    \om_2(l_2-1)c,n_1T,n_2T)}\nonumber\\
     &~~~q^k_{l_1l_2}(\om_1-(l_1-1)c, \om_2-(l_2-1)c)b_{l_1l_2}(\om_1-(l_1-1)c,\om_2-(l_2-1)c,x_1,x_2)d\om_1d\om_2\nonumber\\
     &= \sum_{l_1=1}^{m}\sum_{l_2=1}^{m}\int_{I_{11}}\overline{b_{l_1l_2}(\om_1,\om_2,n_1 T,n_2 T)}q^k_{l_1l_2}(\om_1,\om_2)b_{l_1l_2}(\om_1,\om_2,x_1,x_2)d\om_1d\om_2\nonumber\\
    &=\int_{I_{11}} \e^{-\i\om_1n_1T} \e^{-\j\om_2n_2T} \sum_{l_1=1}^{m}\sum_{l_2=1}^{m}q^k_{l_1l_2}(\om_1,\om_2)b_{l_1l_2}(\om_1,\om_2,x_1,x_2)d\om_1d\om_2.
\end{align}
 which completes the proof.
\end{proof}
\begin{figure*}
  \centering
%  % Requires \usepackage{graphicx}
  \includegraphics[width=15.5cm]{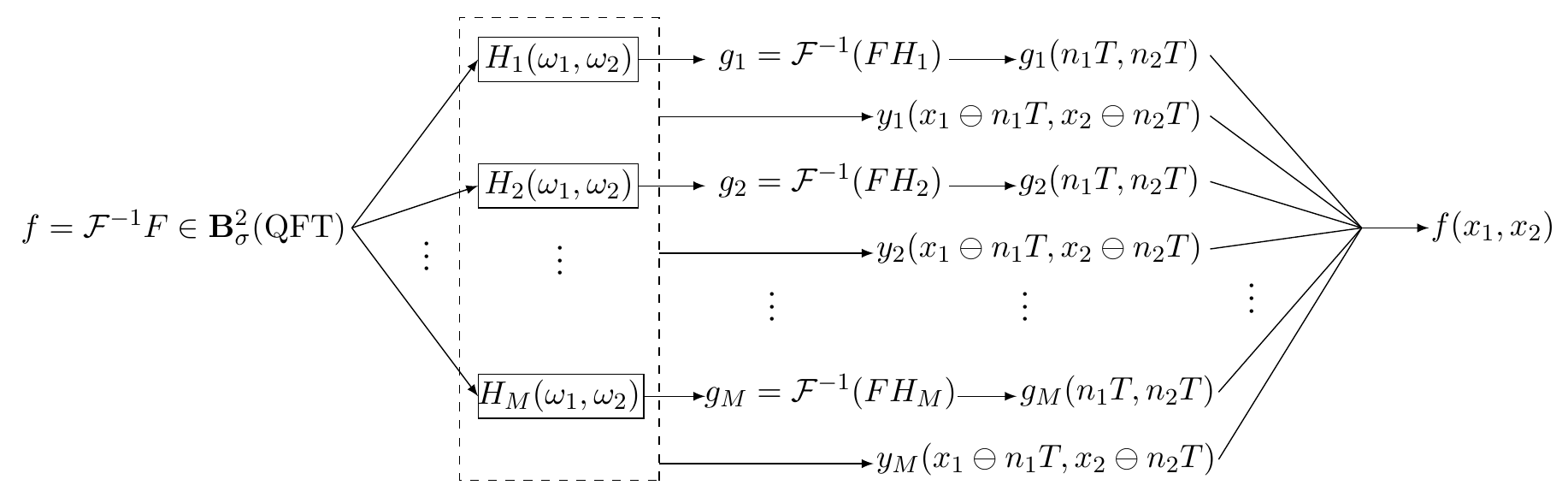}\\
   \caption{Diagram of GSE associated with QFT.}\label{diagram}
\end{figure*}

Now we give the generalized sampling expansion associated with QFT.
\begin{theorem}\label{thm-GSE-QFT}
Let $H_1,H_2,...,H_M$ such that
\begin{enumerate}
  \item $H_k\in  L^2(I,\H)$,
  \item $\underline{H}(\om_1,\om_2)$ is invertible for every $(\om_1,\om_2)\in I_{11}$ and $q^k_{l_1l_2} \in L^2(I_{11},\H)$.
\end{enumerate}
Then $f$ can be reconstructed from samples $g_k(n_1T,n_2T)$ of
\begin{align*}
  g_k(x_1,x_2) &=(f\star h_k)(x_1,x_2)\\
  &=\frac{1}{2\pi} \int_{I}F(\om_1,\om_2)H_k(\om_1,\om_2)\e^{\j\om_2x_2}\e^{\i\om_1x_1}d\om_1d\om_2.
\end{align*}
More specifically,
\begin{equation}\label{GSE-QFT}
f(x_1,x_2)=\sum_{k=1}^{M}\sum_{n_1,n_2 }  g_k(n_1T,n_2T)y_k(x_1\ominus n_1T,x_2\ominus n_2 T)
\end{equation}
where $y_k$ is the inverse QFT of $Y_k$.
\end{theorem}
\begin{proof}
 Since  $\underline{H}(\om_1,\om_2)$ is invertible for every $(\om_1,\om_2) \in I_{11}$  then $$ {\underline{H}^{-1}(\om_1,\om_2)}\overrightarrow{E}(\om_1,\om_2,x_1,x_2)$$ is a $m\times 1$ matrix and the $k$th element equals to $$\sum_{l_1=1}^{m}\sum_{l_2=1}^{m}q^k_{l_1l_2}(\om_1,\om_2)  b_{l_1l_2}(\om_1,\om_2 ,x_1,x_2).$$ Therefore
\begin{align*}
f(x_1,x_2)  &=\frac{1}{2\pi}\int_{I_{11}}\overrightarrow{F}(\om_1,\om_2)\overrightarrow{E}(\om_1,\om_2,x_1,x_2)d\om_1d\om_2\\
     &= \frac{1}{2\pi}\int_{I_{11}}  \overrightarrow{G}(\om_1,\om_2) {\underline{H}(\om_1,\om_2)}^{-1}\overrightarrow{E}(\om_1,\om_2,x_1,x_2)d\om_1d\om_2 \\
   &=   \sum_{k=1}^{M} \frac{1}{2\pi}\int_{I_{11}}\widetilde{G}_k(\om_1,\om_2)  \sum_{l_1=1}^{m}\sum_{l_2=1}^{m}q^k_{l_1l_2}(\om_1,\om_2)  b_{l_1l_2}(\om_1,\om_2 ,x_1,x_2)d\om_1d\om_2.
\end{align*}

 As $F, H_k  \in L^2(I,\H)$, it is easy to see that $\widetilde{G}_k \in L^2(I_{11},\H)$. Also, if $q^k_{l_1l_2} \in L^2(I,\H)$  then $$\sum_{l_1=1}^{m}\sum_{l_2=1}^{m}q^k_{l_1l_2}(\om_1,\om_2)b_{l_1l_2}(\om_1,\om_2,x_1,x_2)\in L^2(I_{11},\H)$$ for every $(x_1,x_2)\in \mathbb{R}^2$.
Since $ \{\frac{T}{2\pi} \e^{-\i \om_1n_1T} \e^{-\j \om_2 n_2T}\}_{(n_1,n_2)\in \mathbb{Z}^2}$ is an orthonormal basis of $L^2(I_{11},\H)$. Therefore by invoking (\ref{Plancherel series}) we have
\begin{align*}
 % \nonumber to remove numbering (before each equation)
    &~~~~ \frac{1}{2\pi} \int_{I_{11}}\widetilde{G}_k(\om_1,\om_2) \sum_{l_1=1}^{m}\sum_{l_2=1}^{m}q^k_{l_1l_2}(\om_1,\om_2)b_{l_1l_2}(\om_1,\om_2,x_1,x_2)d\om_1d\om_2\\
     &= \sum_{n_1,n_2 } \frac{T^2}{8\pi^3}\int_{I_{11}} \widetilde{G}_k(\om_1,\om_2)\e^{\j\om_2n_2T}\e^{\i\om_1n_1T}d\om_1d\om_2  \\
      &~~\int_{I_{11}} \e^{-\i\om_1n_1T} \e^{-\j\om_2n_2T}  \sum_{l_1=1}^{m}\sum_{l_2=1}^{m}q^k_{l_1l_2}(\om_1,\om_2)b_{l_1l_2}(\om_1,\om_2,x_1,x_2)d\om_1d\om_2
    % &=&   \sum_{n_1,n_2 }  g_k(n_1T,n_2T)y_k(x_1\ominus n_1T,x_2\ominus n_2T)
\end{align*}
Hence, by Lemma \ref{samples-g} and Lemma \ref{interpolation func}, we obtain (\ref{GSE-QFT}).
\end{proof}
\section{Examples}\label{S5}

\begin{example}\label{converges fast}
If $f$ is $\sigma$-bandlimited, we have
\begin{equation}\label{classical shannon}
  f(x_1,x_2)=\sum_{n_1,n_2}f(n_1T,n_2T)\frac{\sin(\sigma x_1-n_1 \pi)\sin(\sigma x_2-n_2 \pi)}{(\sigma x_1-n_1 \pi)(\sigma x_2-n_2 \pi)}
\end{equation}
by choosing $m=1$ and $H_1(\om_1,\om_2)=1$, where $T=\frac{\pi}{\sigma}$. Let $\sigma'=\rho \sigma$  with  $\rho>1$ and
$H(\om_1,\om_2)=H^1(\om_1)H^1(\om_2)$
with \begin{equation*}
H^1(\om_1)=\begin{cases}
 1, ~~~~~~|\om_1|\leq \sigma,\\
 0,~~~~~~|\om_1|\geq \sigma',\\
 \frac{1}{(1-\rho)\sigma}|\om_1|+\frac{\rho}{\rho-1},~~\sigma\leq |\om_1|\leq \sigma'.
\end{cases}
\end{equation*}
Note that $f$ is  $\sigma'$-bandlimited. Therefore, by applying Theorem \ref{thm-GSE-QFT} with $M=1$ and $H(\om_1,\om_2)$ defined above, we have
\begin{equation}\label{fast than classical shannon}
 f(x_1,x_2)=\sum_{n_1,n_2}f(n_1T',n_1T')y(x_1- n_1T,x_2- n_2 T)
\end{equation}
where $$y(x_1,x_2)=\frac{4(\sin^2\frac{ \rho\sigma x_1}{2}-\sin^2\frac{\sigma x_1}{2})(\sin^2\frac{ \rho\sigma x_2}{2}-\sin^2\frac{\sigma x_2}{2})}{x_1^2x_2^2\rho^2(\rho-1)^2\sigma^4}$$
and $T'=\frac{\pi}{\sigma'}=\frac{T}{\rho}<T$. For any fixed $(x_1,x_2)\in \mathbb{R}^2$, (\ref{fast than classical shannon}) converges faster than (\ref{classical shannon}). It illustrates that convergence rate of sampling series can be  enhanced by  increasing the  sampling frequency.
\end{example}

\begin{example}
Express a $\sigma$-bandlimited function $f(x_1, x_2)$  from  the samples $g(n_1T,n_2T)$ of the integral
\begin{equation*}
  g(x_1,x_2)= \frac{1}{2\pi}\int_{I}F(\om_1,\om_2)H(\om_1,\om_2)\e^{\j\om_2x_2}\e^{\i\om_1x_1}d\om_1, \om_2
\end{equation*}
where $H(\om_1, \om_2)= {\alpha \beta}{(\beta+\j \om_2)^{-1}(\alpha+\i \om_1)^{-1}}(\alpha,\beta>0)$, $I=[-\sigma,\sigma]^2$, $T=\frac{\pi}{\sigma}$. In fact, by Theorem \ref{CONvo}, we have
$$g(x_1, x_2)=\frac{1}{2\pi}\int_{\mathbb{R}^2}f(y_1,y_2)\widetilde{h}(x_1\ominus y_2, x_2\ominus y_2)dy_1dy_2$$
where
\begin{equation*}
 \widetilde{h}(x_1\ominus y_1, x_2\ominus y_2)=\frac{1}{2\pi}\int_{I}\e^{-\i\om_1y_1}\e^{-\j\om_2y_2}H(\om_1,\om_2) \e^{\j\om_2x_2}\e^{\i\om_1x_1}d\om_1d\om_2.
\end{equation*}
It is easy to see that $H$ satisfies all conditions of Theorem \ref{thm-GSE-QFT}. From (\ref{QFT-interpolation func}), we have
\begin{equation*}
  {Y(\om_1, \om_2)=\frac{T^2(\alpha+\i \om_1)(\beta+\j \om_2)}{2\pi \alpha \beta}\chi_{I}(\om_1,\om_2).}
\end{equation*}
By direct computation,  $4\pi{T^{-2}} y (x_1\ominus n_1T, x_2\ominus n_2T)$ is given by
 \begin{align*}
% \nonumber to remove numbering (before each equation)
    &~   \frac{4\sin{\sigma(x_1-n_1T)}\sin{\sigma(x_2-n_2T)}}{(x_1-n_1T)(x_2-n_2T)} \\ &~  -\frac{4\sin{\sigma(x_2+n_2T)}[\sin{\sigma(x_1-n_1T)}-
    \sigma(x_1-n_1T)\cos{\sigma(x_1-n_1T)}]}{\alpha(x_2+n_2T)(x_1-n_1T)^2} \\
     &~  -\frac{4\sin{\sigma(x_1-n_1T)}[\sin{\sigma(x_2-n_2T)}-
    \sigma(x_2-n_2T)\cos{\sigma(x_2-n_2T)}]}{\beta(x_1-n_1T)(x_2-n_2T)^2} \\
    &~ +\frac{4[\sin{\sigma(x_1-n_1T)}-
    \sigma(x_1-n_1T)\cos{\sigma(x_1-n_1T)}]}{\alpha\beta(x_1-n_1T)^2(x_2+n_2T)^2} \\ &~~~~~ [\sin{\sigma(x_2+n_2T)}-
    \sigma(x_2+n_2T)\cos{\sigma(x_2+n_2T)}].
\end{align*}
\end{example}
\begin{example}
Theorem \ref{thm-GSE-QFT} permits us to express a $\sigma$-bandlimited function $f(x_1, x_2)$   from  its samples and  samples of its partial derivatives. Let $m=2$, $T=\frac{2\pi}{\sigma}$, $c=\sigma$ and $H_1(\om_1, \om_2)=1$, $ H_2(\om_1, \om_2)=\i \om_1$, $H_3(\om_1, \om_2)=\j \om_2$, $H_4(\om_1, \om_2)= \k \om_1 \om_2$. By [], we have $g_1(x_1, x_2)=f(x_1, x_2)$, $g_2(x_1, x_2)= \frac{\partial f}{\partial x_1}(x_1,-x_2)$, $g_3(x_1, x_2)=\frac{\partial f}{\partial x_2}(x_1,x_2)$, $g_4(x_1, x_2)=-\frac{\partial^2 f}{\partial x_1x_2}(x_1,-x_2)$.
Furthermore,
$$\underline{H}=\begin{pmatrix} 1&\i \om_1&\j \om_2 & \k \om_1\om_2\\ 1&\i \om_1&\j (\om_2+c) &\k \om_1(\om_2+c) \\ 1&\i (\om_1+c)&\j \om_2 &  \k( \om_1+c)\om_2 \\ 1&\i (\om_1+c)&\j (\om_2+c) & \k( \om_1+c)(\om_2+c) \end{pmatrix}=A_1+A_2\j$$
where $A_1=\begin{pmatrix} 1&\i \om_1&0 & 0\\ 1&\i \om_1&0 &0 \\ 1&\i (\om_1+c)&0 &  0\\ 1&\i (\om_1+c)&0& 0 \end{pmatrix}$ and $A_2=\begin{pmatrix} 0&0&\om_2 & \i \om_1\om_2\\ 0&0& (\om_2+c) &\i \om_1(\om_2+c) \\ 0&0& \om_2 &  \i( \om_1+c)\om_2 \\ 0&0& (\om_2+c) & \i( \om_1+c)(\om_2+c) \end{pmatrix}.$
The complex adjoint matrix \cite{zhang1997quaternions} of $\underline{H}$ denoted by $C(\underline{H})$ is defined   as
\begin{equation*}
 C(\underline{H})=\begin{pmatrix}A_1&A_2\\-\overline{A_2} &\overline{A_1 }\end{pmatrix}.
\end{equation*}
Zhang \cite{zhang1997quaternions} showed that  $\underline{H}$ is invertible if and only if $ C(\underline{H})$ is invertible. Moreover, $C(\underline{H}^{-1})=[C(\underline{H})]^{-1}$ if $\underline{H}^{-1}$ exsits. Since $|C(\underline{H})]|=c^8\neq 0$ for every $(\om_1,\om_2)\in I$. Therefore $\underline{H}$ is invertible for every $(\om_1,\om_2)\in I$. In fact, $[C(\underline{H})]^{-1}=\left(U_1,U_2\right)$  where
\begin{equation*}
U_1={\left(
\begin{array}{cccccccc}
 \frac{(c+\om_2) (c+\om_1)}{c^2} & -\frac{\om_2 (c+\om_1)}{c^2} & -\frac{(c+\om_2) \om_1}{c^2} & \frac{\om_1 \om_2}{c^2}   \\
 \frac{\i (c+\om_2)}{c^2} & -\frac{\i \om_2}{c^2} & -\frac{\i (c+\om_2)}{c^2} & \frac{\i \om_2}{c^2} \\
 0 & 0 & 0 & 0 \\
 0 & 0 & 0 & 0   \\
 0 & 0 & 0 & 0   \\
 0 & 0 & 0 & 0  \\
 -\frac{c+\om_1}{c^2} & \frac{c+\om_1}{c^2} & \frac{\om_1}{c^2} & -\frac{\om_1}{c^2} \\
 -\frac{\i}{c^2} & \frac{\i}{c^2} & \frac{\i}{c^2} & -\frac{\i}{c^2}  \\
\end{array}
\right)}
\end{equation*}
and
\begin{equation*}
U_2={\left(
\begin{array}{cccccccc}
   0 & 0 & 0 & 0 \\
  0 & 0 & 0 & 0 \\
   \frac{c+\om_1}{c^2} & -\frac{c+\om_1}{c^2} & -\frac{\om_1}{c^2} & \frac{\om_1}{c^2} \\
     -\frac{\i}{c^2} & \frac{\i}{c^2} & \frac{\i}{c^2} & -\frac{\i}{c^2} \\
 \frac{(c+\om_2) (c+\om_1)}{c^2} & -\frac{\om_2 (c+\om_1)}{c^2} & -\frac{(c+\om_2) \om_1}{c^2} & \frac{\om_2 \om_1}{c^2} \\
  -\frac{\i (c+\om_2)}{c^2} & \frac{\i \om_2}{c^2} & \frac{\i (c+\om_2)}{c^2} & -\frac{\i v}{c^2} \\
 0 & 0 & 0 & 0 \\
  0 & 0 & 0 & 0 \\
\end{array}
\right)}.
\end{equation*}
Thus
\begin{equation*}
  \underline{H}^{-1}=\left(
\begin{array}{cccc}
 \frac{(c+\om_2) (c+\om_1)}{c^2} & -\frac{\om_2 (c+\om_1)}{c^2} & -\frac{(c+\om_2) \om_1}{c^2} & \frac{\om_2 \om_1}{c^2} \\
 \frac{\i (c+\om_2)}{c^2} & -\frac{\i \om_2}{c^2} & -\frac{\i (c+\om_2)}{c^2} & \frac{\i \om_2}{c^2} \\
 \frac{\j(c+\om_1)}{c^2} & -\frac{\j(c+\om_1)}{c^2} & -\frac{\j\om_1}{c^2} & \frac{\j\om_1}{c^2} \\
 -\frac{\k}{c^2} & \frac{\k}{c^2} & \frac{\k}{c^2} & -\frac{\k}{c^2} \\
\end{array}
\right)
\end{equation*}
and hence
\begin{align*}
% \nonumber to remove numbering (before each equation)
&~~~~ \frac{1}{2\pi}Y_1(\om_1,\om_2)\\
&= \sigma^{-4}(c+\om_2) (c+\om_1)\chi_{I_{11}}(\om_1,\om_2)- \sigma^{-4}(\om_2-c) (c+\om_1)\chi_{I_{12}}(\om_1,\om_2)\\
    &~~ -  \sigma^{-4}(c+\om_2) (\om_1-c)\chi_{I_{21}}(\om_1,\om_2)+\sigma^{-4}(\om_2-c) (\om_1-c)\chi_{I_{22}}(\om_1,\om_2).
\end{align*}
By direct computation,
\begin{equation*}
  y_1(x_1\ominus n_1T,x_2\ominus n_2T)=\frac{16
  \sin^2{(\frac{\sigma}{2}x_1-n_1\pi)}\sin^2{(\frac{\sigma}{2}x_2-n_2\pi)}}{\sigma^4(x_1-n_1T)^2(x_2-n_2T)^2}.
\end{equation*}
Similarly, we have
\begin{align*}
% \nonumber to remove numbering (before each equation)
   y_2(x_1\ominus n_1T,x_2\ominus n_2T) &=  \frac{16
  \sin^2{(\frac{\sigma}{2}x_1-n_1\pi)}\sin^2{(\frac{\sigma}{2}x_2+n_2\pi)}}{\sigma^4(x_1-n_1T)(x_2+n_2T)^2}, \\
   y_3(x_1\ominus n_1T,x_2\ominus n_2T) &= \frac{16
  \sin^2{(\frac{\sigma}{2}x_1-n_1\pi)}\sin^2{(\frac{\sigma}{2}x_2-n_2\pi)}}{\sigma^4(x_1-n_1T)^2(x_2-n_2T)}, \\
   y_4(x_1\ominus n_1T,x_2\ominus n_2T) &=  -\frac{16
  \sin^2{(\frac{\sigma}{2}x_1-n_1\pi)}\sin^2{(\frac{\sigma}{2}x_2+n_2\pi)}}{\sigma^4(x_1-n_1T)(x_2+n_2T)}.
\end{align*}
 \end{example}

 \section{Sampling theorem for Quaternion linear canonical transform}\label{S6}

 The right-sided quaternion linear canonical transform (QLCT)  which is generalization of linear canonical transform (LCT)
 to quaternion algebra, was firstly studied in  \cite{kou2013uncertainty}. In this section, we  investigate the sampling theory associated with QLCT. The  right-sided QLCT of a signal $f\in \La$ with real matrix parameter $A_i=\begin{pmatrix} a_i&b_i \\ c_i&d_i \end{pmatrix}\in \mathbb{R}^{2\times2}$   such that $\det{(A_i)}=1$ for $i=1,2$  is defined by \cite{kou2013uncertainty}
 \begin{equation*}
   (\mathcal{L}f)(\om_1,\om_2) :=\int_{\mathbb{R}^2} f(x_1,x_2)K_{A_1}^{\i}(x_1,\om_1)K_{A_2}^{\j}(x_2,\om_2)dx_1dx_2
 \end{equation*}
 where
\begin{equation}\label{kernel1}
  K_{A_1}^{\i}(x_1,\om_1):=\frac{1}{\sqrt{\i  2\pi b_1}}\e^{\i(\frac{a_1}{2b_1}x_1^2-\frac{1}{b_1}x_1\om_1+\frac{d_1}{2b_1}w_1^2)},~~~ \text{for}~~ b_1\neq0
\end{equation}
and
\begin{equation}\label{kernel2}
  K_{A_2}^{\j}(x_2,\om_2):=\frac{1}{\sqrt{\j 2\pi b_2}}\e^{\j(\frac{a_2}{2b_2}x_2^2-\frac{1}{b_2}x_2\om_2+\frac{d_2}{2b_2}w_2^2)},~~~ \text{for}~~ b_2\neq0.
\end{equation}
Here, $\frac{1}{\sqrt{\u  2\pi b }}$  represents $|2\pi b|^{\frac{-1}{2}} \e^{\u \frac{  \mathrm{sgn} b-2}{4}\pi}$ for any pure imaginary unit quaternion $\u$ and nonzero real number $b$.

 If $\mathcal{L}f$ is also in $\La$,  the the inversion QLCT formula \cite{cheng2016properties} holds,  that is
\begin{equation*}
f(x_1,x_2)=(\mathcal{L}^{-1}\mathcal{L}f)(x_1,x_2):=  \int_{\mathbb{R}^2}(\mathcal{L}f)(\om_1,\om_2) K_{A_2^{-1}}^{\j}(\om_2,x_2) K_{A_1^{-1}}^{\i}(\om_1,x_1)d\omega_1d\omega_2,
\end{equation*}
for almost every $(x_1,x_2)\in \mathbb{R}^2$.
 By Plancherel theorem \cite{cheng2016properties}, the QLCT can be extended to $\Lb$. As an operator on $\Lb$, the QLCT $\Phi$ is a
 bijection and  the  Parseval's identity $\|\Phi f\|_2=\|f\|_2$ holds. Since $\Phi(\Phi^{-1})$ coincides with  $ \mathcal{L}(\mathcal{L}^{-1})$ in $\La\cap \Lb$.  So we use letter $F_{\mathbf{A}}$ to denote the QLCT of $f\in \La \cup \Lb$.

 In the classical case, the LCT is just a variation of the standard Fourier transform, some of
 its properties can be deduced from those of the Fourier transform by a change of variable. Moreover, the proofs of many sampling formulae associated with the LCT are somewhat based on those of the Fourier transform. In the quaternionic case,  however,  $\mathcal{L}$  can not directly establish relation with $\mathcal{F}$ as mentioned in \cite{cheng2016properties}. So it is hard to derive GSE associated with the QLCT from existing  results of the QFT. On the other hand,   Lemma \ref{samples-g} is based on periodicity of kernel $\e^{\j\om_2x_2}\e^{\i\om_1x_1}$, but due to the  quadratic term of kernel (\ref{kernel1})   and  (\ref{kernel2}) in the QLCT, periodicity of kernel  no longer possess.   When $M>1$, (\ref{GSE-QFT}) of Theorem \ref{thm-GSE-QFT} is also called \emph{multichannel sampling expansion}.  In the following, we give a single-channel sampling expansion associated with the QLCT, that is we only focus on the case of $M=1$.

 Firstly, we introduce the generalized translation  related to the QLCT:
 \begin{align}\label{eneralized translation  related to the QLCT}
  &~~~ ¡¡f(x_1\boxminus y_1,x_2 \boxminus y_2)\\
  &:=\int_{\mathbb{R}^2} \overline{K_{A_2^{-1}}^{\j}(\om_2,y_2) K_{A_1^{-1}}^{\i}(\om_1,y_1)} F_{\mathbf{A}} (\om_2,\om_2)K_{A_2^{-1}}^{\j}(\om_2,x_2) K_{A_1^{-1}}^{\i}(\om_1,x_1)d\om_1d\om_2
 \end{align}
 provided taht the right-hand side integral is well defined. Then we have the following theorem.
 \begin{theorem}\label{thm-GSE-QLCT}
 Suppose that $f$ is $\sigma$-bandlimited in QLCT sense, that is
 \begin{equation*}
  f(x_1,x_2)= \int_{I}F_{\mathbf{A}} (\om_2,\om_2)K_{A_2^{-1}}^{\j}(\om_2,x_2) K_{A_1^{-1}}^{\i}(\om_1,x_1)d\om_1d\om_2
\end{equation*}
where $F_{\mathbf{A}}\in L^2(I,\H)$ and $I=[-\sigma,\sigma]^2$. Let
$$g(x_1,x_2)=\int_{I}F_{\mathbf{A}} (\om_2,\om_2) H(\om_1,\om_2)K_{A_2^{-1}}^{\j}(\om_2,x_2) K_{A_1^{-1}}^{\i}(\om_1,x_1)d\om_1d\om_2$$
 where $H(\om_1,\om_2), H(\om_1,\om_2)^{-1} \in L^2(I,\H)$.
Let $$Y_{\mathbf{A}}(\om_1,\om_2)=T^2 |b_1b_2| H(\om_1,\om_2)^{-1}\chi_{I}(\om_1,\om_2).$$
Then $f$ can be reconstructed from samples $g (n_1b_1T,n_2b_2T)$:
 \begin{equation}\label{GSE-QLCT}
\addtocounter{equation}{1}
f(x_1,x_2)= \sum_{n_1,n_2 }  g(n_1b_1T,n_2b_2T)y (x_1\boxminus n_1b_1T,x_2\boxminus n_2 b_2T)
\end{equation}
where $y $ is the inverse QLCT of $Y_{\mathbf{A}}$ and $T=\frac{ \pi}{\sigma}$.
\end{theorem}
\begin{proof}
We note that $\varphi_{n_1n_2}(\om_1,\om_2):=T |b_1b_2|^{\frac{1}{2}}\overline{K_{A_2^{-1}}^{\j}(\om_2,n_2b_2T) K_{A_1^{-1}}^{\i}(\om_1,n_1b_1T)}$ is an orthonormal basis of $L^2(I ,\H)$.
Therefore by invoking (\ref{Plancherel series}) we have
\begin{align*}
 % \nonumber to remove numbering (before each equation)
  f(x_1,x_2)   &=  \int_{I}F_{\mathbf{A}} (\om_2,\om_2)K_{A_2^{-1}}^{\j}(\om_2,x_2) K_{A_1^{-1}}^{\i}(\om_1,x_1)d\om_1d\om_2\\
  &= \int_{I}F_{\mathbf{A}} (\om_2,\om_2)H(\om_1,\om_2) H(\om_1,\om_2)^{-1}K_{A_2^{-1}}^{\j}(\om_2,x_2) K_{A_1^{-1}}^{\i}(\om_1,x_1)d\om_1d\om_2\\
     &= \sum_{n_1,n_2 } \left( \int_{I}F_{\mathbf{A}} (\om_2,\om_2)H(\om_1,\om_2)\overline{\varphi_{n_1n_2}(\om_1,\om_2)}d\om_1d\om_2 \right) \\  &~~\left (\int_{I}\varphi_{n_1n_2}(\om_1,\om_2) H(\om_1,\om_2)^{-1}K_{A_2^{-1}}^{\j}(\om_2,x_2) K_{A_1^{-1}}^{\i}(\om_1,x_1)d\om_1d\om_2 \right)  \\
     &= \sum_{n_1,n_2 } g(n_1b_1T,n_2b_2T)y (x_1\boxminus n_1b_1T,x_2\boxminus n_2 b_2T)
    % &=&   \sum_{n_1,n_2 }  g_k(n_1T,n_2T)y_k(x_1\ominus n_1T,x_2\ominus n_2T)
\end{align*}
which completes the proof.
 \end{proof}

 Now we give an example for this Theorem.  Suppose that  $f$ is $\sigma$-bandlimited in QLCT sense and let $H(\om_1,\om_2)=1$, then $g(x_1,x_2)=f(x_1,x_2)$. Therefore
  \begin{equation*}
   f(x_1,x_2)= \sum_{n_1,n_2 }  f(n_1b_1T,n_2b_2T)y (x_1\boxminus n_1b_1T,x_2\boxminus n_2 b_2T)
  \end{equation*}

  By (\ref{eneralized translation  related to the QLCT}) we obtain $y (x_1\boxminus n_1b_1T,x_2\boxminus n_2 b_2T)=\varrho_1+\varrho_2 \varrho_3$ where $\varrho_1, \varrho_2, \varrho_3$, respectively, are
 \begin{equation*}
   \frac{T^2b_1|b_2|}{\pi^2}\cos \left(\frac{a_2b_2n_2^2T^2}{2}-\frac{a_2x_2^2}{2b_2} \right)\frac{\sin (n_1\pi-\frac{\pi x_1^2}{b_1T})\sin (n_2\pi-\frac{\pi x_2^2}{b_2T})}{(n_1b_1T-x_1)(n_2b_2T-x_2)}\e^{\i (\frac{a_1b_1n_1^2T^2}{2}-\frac{a_1x_1^2}{2b_1})},
 \end{equation*}
 \begin{equation*}
     \erf \left(\frac{|b_1b_2|^\frac{-1}{2}}{2T}(2d_1\pi- b_1n_1T^2-x_1T)\right)+\erf \left(\frac{|b_1b_2|^\frac{-1}{2}}{2T}(2d_1\pi+b_1n_1T^2+x_1T)\right)
 \end{equation*}
 and
 \begin{equation*}
 \left( \frac{|b_1|}{\pi}\right)^\frac{3}{2}\frac{T^2|d_1|^{\frac{-1}{2}}}{4}\sin \left(\frac{a_2b_2n_2^2T^2}{2}-\frac{a_2x_2^2}{2b_2}\right)\frac{\sin(n_2\pi-\frac{\pi x_2^2}{b_2T})}{n_2b_2T-x_2} \e^{\i(\frac{a_1b_1n_1^2T^2}{2}+\frac{a_1x_1^2}{2b_1}-\frac{(b_1n_1T+x_1)^2}{4b_1d_1})}\j.
 \end{equation*}

\section{Conclusion}\label{S7}

In this paper, we introduced the GSE associated with QFT. The GSE formula illustrates how a bandlimited quaternion valued signal can be recovered from the samples of  system output signals. This has been realized by taking  advantage of   generalized translation
and convolution. Moreover, we have further discussed  the sampling formula for $\sigma$-bandlimited  quaternion valued signal in quaternion linear canonical transform sense.
 %The GSE  of quaternion valued signals to  vector-signal and vector-image  processing is what  the  GSE of complex or real valued signals to classical signal processing.

\section{Acknowledgements}\label{S7}
The authors acknowledge financial support from the National Natural Science Foundation of China under Grant (No. 11401606), University of Macau (No. MYRG2015-00058-L2-FST and No. MYRG099(Y1-L2)-FST13-KKI) and the Macao Science and Technology
Development Fund (No. FDCT/094/2011/A and No. FDCT/099/2012/A3).

\end{document}